\documentclass[reqno,11pt]{article}
\usepackage{graphicx, amsmath, amsthm,amssymb}

\input{stochbif.sty}

\begin{document}


\title{Kramers' law:\\ Validity, derivations and generalisations}
\author{Nils Berglund} 
\date{June 28, 2011. Updated, January 22, 2013}   

\maketitle

\begin{abstract}
Kramers' law describes the mean transition time of an overdamped
Brownian particle between local minima in a potential landscape. We review
different approaches that have been followed to obtain a mathematically rigorous
proof of this formula. We also discuss some generalisations, and a case in which
Kramers' law is not valid. This review is written for both mathematicians and
theoretical physicists, and endeavours to link concepts and terminology from
both fields. 
\end{abstract}



\leftline{\small 2000 {\it Mathematical Subject Classification.\/} 
58J65, 60F10, (primary), 60J45, 34E20 (secondary)}


\section{Introduction}
\label{sec_intro}

The overdamped motion of a Brownian particle in a potential $V$ is governed by
a first-order Langevin (or Smoluchowski) equation, usually written in the
physics literature as 
\begin{equation}
 \label{in02}
\dot x = -\nabla V(x) + \sqrt{2\eps}\,\xi_t\;, 
\end{equation} 
where $\xi_t$ denotes zero-mean, delta-correlated Gaussian white noise. 
We will rather adopt the mathematician's notation, and write~\eqref{in02} as
the It\^o stochastic differential equation 
\begin{equation}
 \label{in01}
\6x_t = -\nabla V(x_t)\6t + \sqrt{2\eps} \6W_t\;, 
\end{equation} 
where $W_t$ denotes $d$-dimensional Brownian motion. 
The potential is a function $V:\R^d\to\R$, which we will always assume to be
smooth and growing sufficiently fast at infinity. 

The fact that the drift term in~\eqref{in01} has gradient form entails two
important properties, which greatly simplify the analysis: 
\begin{enum}
\item	There is an invariant probability measure, with the explicit expression 
\begin{equation}
 \label{in03}
\mu(\6x) = \frac1Z \e^{-V(x)/\eps} \6x\;, 
\end{equation} 
where $Z$ is the normalisation constant.

\item	The system is \defwd{reversible}\/ with respect to the invariant measure
$\mu$, that is, the transition probability density satisfies the \defwd{detailed
balance condition} 
\begin{equation}
 \label{in04}
p(y,t|x,0) \e^{-V(x)/\eps} = p(x,t|y,0) \e^{-V(y)/\eps}\;. 
\end{equation} 
\end{enum}

The main question we are interested in is the following. Assume that the
potential $V$ has several (meaning at least two) local minima. How long does
the Brownian particle take to go from one local minimum to another one?

To be more precise, let $x^\star$ and $y^\star$ be two local minima of $V$, and
let $\cB_\delta(y^\star)$ be the ball of radius $\delta$ centred in $y^\star$,
where $\delta$ is a small positive constant (which may possibly depend on
$\eps$). We are interested in characterising the first-hitting time of this
ball, defined as the random variable 
\begin{equation}
 \label{in05}
\tau^{x^\star}_{y^\star} = \inf\setsuch{t>0}{x_t\in\cB_\delta(y^\star)}
\qquad
\text{where $x_0=x^\star$\;.}
\end{equation} 
The two points $x^\star$ and $y^\star$ being local minima, the potential along
any continuous path $\gamma$ from $x^\star$ to $y^\star$ must increase and
decrease again, at least once but possibly several times. We can determine the
maximal value of $V$ along such a path, and then minimise this value over all
continuous paths from $x^\star$ to $y^\star$. This defines a
\defwd{communication height}
\begin{equation}
 \label{in06}
H(x^\star,y^\star) = \inf_{\gamma: x^\star\to y^\star} 
\biggpar{\sup_{z\in\gamma} V(z)}\;. 
\end{equation} 
Although there are many paths realising the infimum in~\eqref{in06}, the
communication height is generically reached at a unique point $z^\star$, which
we will call the \defwd{relevant saddle} between $x^\star$ and~$y^\star$. In
that case, $H(x^\star,y^\star)=V(z^\star)$ (see \figref{fig_doublewell}). One
can show that generically, $z^\star$ is a critical point of index $1$ of the
potential, that is, when seen from $z^\star$ the potential decreases in one
direction and increases in the other $d-1$ directions. This translates
mathematically into $\nabla V(z^\star)=0$ and the Hessian $\hessian V(z^\star)$
having exactly one strictly negative and $d-1$ strictly positive eigenvalues. 

\begin{figure}
\centerline{\includegraphics*[clip=true,height=50mm]{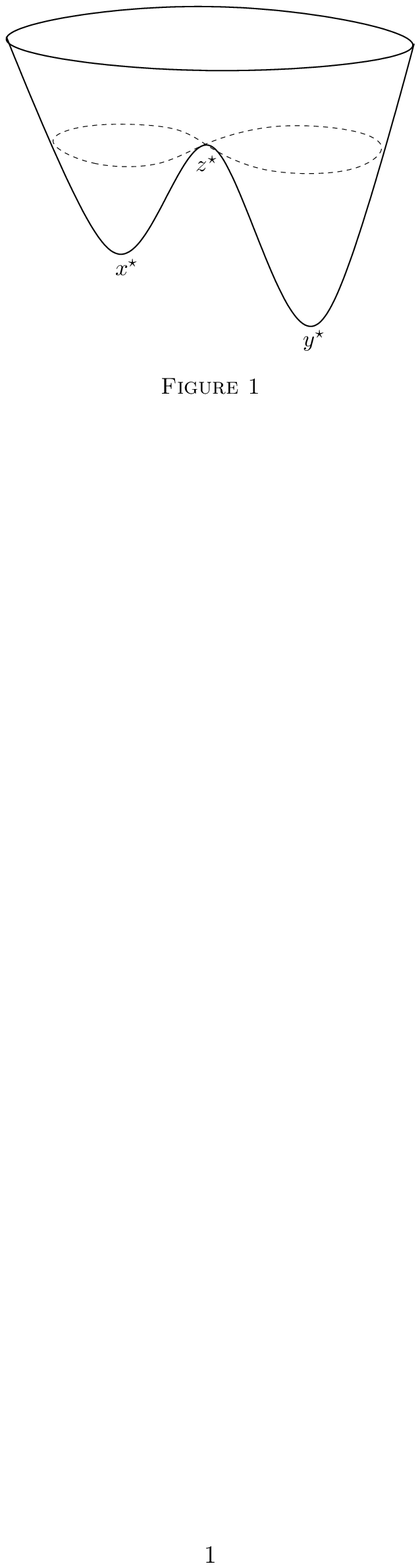}}
 \vspace{2mm}
\caption[]{Graph of a potential $V$ in dimension $d=2$, with two local minima
$x^\star$ and $y^\star$ and saddle $z^\star$. 
}
\label{fig_doublewell}
\end{figure}

In order to simplify the presentation, we will state the main results in the
case of a double-well potential, meaning that $V$ has exactly two local minima
$x^\star$ and $y^\star$, separated by a unique saddle $z^\star$
(\figref{fig_doublewell}), henceforth referred to as \lq\lq the double-well
situation\rq\rq. The Kramers law has been extended to potentials with
more than two local minima, and we will comment on its form in these cases in
Section~\ref{ssec_potential} below. 

In the context of chemical reaction rates, a relation for the mean transition
time $\tau^{x^\star}_{y^\star}$ was first proposed by van t'Hoff, and later
physically justified
by Arrhenius~\cite{Arrhenius}. It reads 
\begin{equation}
 \label{in07} 
\expec{\tau^{x^\star}_{y^\star}} \simeq C \e^{[V(z^\star)-V(x^\star)]/\eps}\;.
\end{equation} 
The Eyring--Kramers law~\cite{Eyring,Kramers} is a refinement of Arrhenius'
law, as it gives an approximate value of the prefactor $C$ in~\eqref{in07}. 
Namely, in the one-dimensional case $d=1$, it reads 
\begin{equation}
 \label{in08}
 \expec{\tau^{x^\star}_{y^\star}} \simeq
\frac{2\pi}{\sqrt{V''(x^\star)\abs{V''(z^\star)}}}
\e^{[V(z^\star)-V(x^\star)]/\eps}\;,
\end{equation} 
that is, the prefactor depends on the curvatures of the potential at the
starting minimum $x^\star$ and at the saddle $z^\star$. Smaller curvatures lead
to longer transition times. 

In the multidimensional case $d\geqs 2$, the Eyring--Kramers law reads 
\begin{equation}
 \label{in09}
  \expec{\tau^{x^\star}_{y^\star}} \simeq \frac{2\pi}{\abs{\lambda_1(z^\star)}}
\sqrt{\frac{\abs{\det(\hessian V(z^\star))}}{\det(\hessian V(x^\star))}}
\e^{[V(z^\star)-V(x^\star)]/\eps}\;, 
\end{equation} 
where $\lambda_1(z^\star)$ is the single negative eigenvalue of the Hessian
$\hessian V(z^\star)$. If we denote the eigenvalues of $\hessian V(z^\star)$ by
$\lambda_1(z^\star)<0<\lambda_2(z^\star)\leqs\dots\leqs\lambda_d(z^\star)$, and
those of $\hessian V(x^\star)$ by
$0<\lambda_1(x^\star)\leqs\dots\leqs\lambda_d(x^\star)$, the
relation~\eqref{in10} can be rewritten as 
\begin{equation}
 \label{in10}
   \expec{\tau^{x^\star}_{y^\star}} \simeq 2\pi
\sqrt{\frac{\lambda_2(z^\star)\dots\lambda_d(z^\star)}{\abs{\lambda_1(z^\star)}
\lambda_1(x^\star)\dots\lambda_d(x^\star) } }
\e^{[V(z^\star)-V(x^\star)]/\eps}\;,
\end{equation} 
which indeed reduces to~\eqref{in08} in the case $d=1$. Notice that for
$d\geqs2$, smaller curvatures at the saddle in the stable directions (a \lq\lq
broader mountain pass\rq\rq) decrease the mean transition time, while a smaller
curvature in the unstable direction increases it. 

The question we will address is whether, under which assumptions and for which
meaning of the symbol $\simeq$ the Eyring--Kramers law~\eqref{in09} is true. 
Answering this question has taken a surprisingly long time, a full proof
of~\eqref{in09} having been obtained only in 2004~\cite{BEGK}. 

In the sequel, we will present several approaches towards a rigorous proof of
the Arrhenius and Eyring--Kramers laws. In Section~\ref{sec_ld}, we present the
approach based on the theory of large deviations, which allows to prove
Arrhenius' law for more general than gradient systems, but fails to control the
prefactor. In Section~\ref{sec_analytic}, we review different analytical
approaches, two of which yield a full proof of~\eqref{in09}. Finally, in
Section~\ref{sec_gen}, we discuss some situations in which the classical
Eyring--Kramers law does not apply, but either admits a generalisation, or has
to be replaced by a different expression. 

\subsection*{Acknowledgements:}
This review is based on a talk given at the meeting \lq\lq Inhomogeneous Random
Systems\rq\rq at Institut Henri Poincar\'e, Paris, on Januray 26, 2011. It is a
pleasure to thank Christian Maes for inviting me, and Fran\c cois Dunlop, 
Thierry Gobron and Ellen Saada for organising the meeting. I'm also grateful to
Barbara Gentz for numerous discussions and useful comments on the manuscript,
and to Aur\'elien Alvarez for sharing his knowledge of Hodge theory. 

\newpage


\section{Large deviations and Arrhenius' law}
\label{sec_ld}

The theory of large deviations has applications in many fields of
probability~\cite{DZ,DS}. It allows in particular to give a mathematically
rigorous framework to what is known in physics as the path-integral approach,
for a general class of stochastic differential equations of the form
\begin{equation}
 \label{ld01}
\6x_t = f(x_t)\6t + \sqrt{2\eps}\6W_t\;, 
\end{equation} 
where $f$ need not be equal to the gradient of a potential $V$ (it is even 
possible to consider an $x$-dependent diffusion coefficient
$\sqrt{2\eps}\,g(x_t)\6W_t$). In this
context, a \defwd{large-deviation principle}\/ is a relation stating that for
small $\eps$, the probability of sample paths being close to a function
$\varphi(t)$ behaves like 
\begin{equation}
 \label{ld02}
\bigprob{x_t\simeq\varphi(t), 0\leqs t\leqs T} 
\simeq \e^{-I(\varphi)/2\eps}
\end{equation} 
(see~\eqref{ld04} below for a mathematically precise formulation). The quantity
$I(\varphi)=I_{[0,T]}(\varphi)$ is called \defwd{rate function}\/ or
\defwd{action functional}. Its
expression was determined by Schilder~\cite{Schilder1966} in the case $f=0$ of
Brownian motion, using the Cameron--Martin--Girsanov formula. Schilder's result
has been extended to general equations of the form~\eqref{ld01} by  
Wentzell and Freidlin~\cite{VF70}, who showed that 
\begin{equation}
 \label{ld03}
I(\varphi) = \frac12 \int_0^T \norm{\dot\varphi(t) - f(\varphi(t))}^2 \6t\;. 
\end{equation} 
Observe that $I(\varphi)$ is nonnegative, and vanishes if and only if
$\varphi(t)$ is a solution of the deterministic equation
$\dot\varphi=f(\varphi)$. One may think of the rate function 
as representing the cost of tracking the function $\varphi$ rather than
following the deterministic dynamics. 

A precise formulation of~\eqref{ld02} is that for any set $\Gamma$ of paths
$\varphi:[0,T]\to\R^d$, one has 
\begin{equation}
 \label{ld04}
-\inf_{\Gamma^\circ} I 
\leqs \liminf_{\eps\to0} 2\eps \log \bigprob{(x_t)\in\Gamma} 
\leqs \limsup_{\eps\to0} 2\eps \log \bigprob{(x_t)\in\Gamma} 
\leqs -\inf_{\overline\Gamma} I\;.
\end{equation}
For sufficiently well-behaved sets of paths $\Gamma$, the infimum of the rate
function over the interior $\Gamma^\circ$ and the closure $\overline\Gamma$
coincide, and thus 
\begin{equation}
 \label{ld05}
 \lim_{\eps\to0} 2\eps \log \bigprob{(x_t)\in\Gamma}
= -\inf_{\Gamma} I \;.
\end{equation} 
Thus roughly speaking, we can write
$\prob{(x_t)\in\Gamma}\simeq\e^{-\inf_{\Gamma} I/2\eps}$, but we should keep
in mind that this is only true in the sense of logarithmic
equivalence~\eqref{ld05}. 

\begin{remark}
\label{remark_ldp}
The large-deviation principle~\eqref{ld04} can be considered as an
infinite-dimen\-sional version of Laplace's method. In the finite-dimensional
case of functions $w:\R^d\to\R$, Laplace's method yields 
\begin{equation}
 \label{ld06}
\lim_{\eps\to0} 2\eps \log 
\int_\Gamma \e^{-w(x)/2\eps}\6x 
= -\inf_\Gamma w\;,
\end{equation} 
and also provides an asymptotic expansion for the prefactor $C(\eps)$ such that
\begin{equation}
 \label{ld06B}
\int_\Gamma \e^{-w(x)/2\eps}\6x 
= C(\eps) \e^{-\inf_\Gamma w/2\eps}\;.
\end{equation}  
This approach can be extended formally to the infinite-dimensional case, where
it yields to a Hamilton--Jacobi equation, which is often used to derive
subexponential corrections to large-deviation results via a (see
e.g.~\cite{MS3}). We are not aware, however, of this procedure having been
justified mathematically. 
\end{remark}

\begin{figure}
\centerline{\includegraphics*[clip=true,height=50mm]{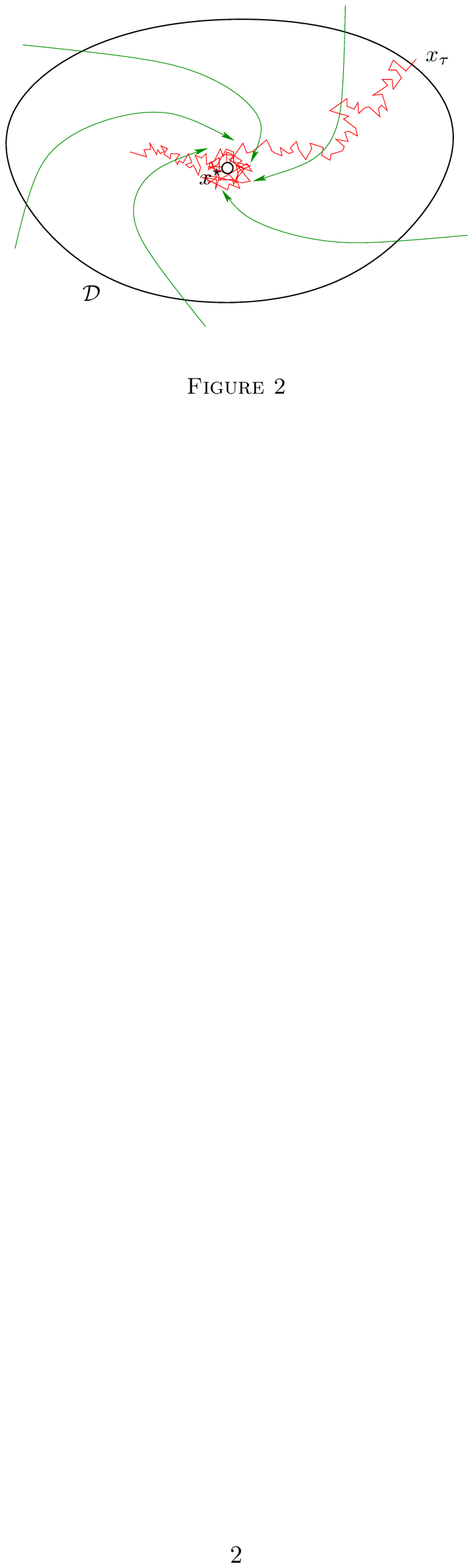}}
 \vspace{2mm}
\caption[]{The setting of Theorems~\ref{thm_FW} and~\ref{thm_Day}. The domain
$\cD$ contains a unique stable equilibrium point $x^\star$, and all orbits of
the deterministic system $\dot{x}=f(x)$ starting in $\overline\cD$ converge to
$x^\star$. 
}
\label{fig_stable}
\end{figure}

Let us now explain how the large-deviation principle~\eqref{ld04} can be
used to prove Arrhenius' law. Let $x^\star$ be a stable equilibrium point 
of the deterministic system $\dot x=f(x)$. In the gradient case $f=-\nabla V$,
this means that $x^\star$ is a local minimum of $V$. Consider a domain
$\cD\subset\R^d$ whose closure is included in the domain of attraction of
$x^\star$ (all orbits of $\dot x=f(x)$ starting in $\overline\cD$ converge to
$x^\star$, see \figref{fig_stable}).
The \defwd{quasipotential}\/ is the function defined for $z\in\overline\cD$ by 
\begin{equation}
 \label{ld07}
\Vbar(z) = \inf_{T>0} 
\;\inf_{\varphi:\varphi(0)=x^\star,\varphi(T)=z}
I(\varphi)\;. 
\end{equation} 
It measures the cost of reaching $z$ in arbitrary time. 

\begin{theorem}[\cite{VF69,VF70}]
\label{thm_FW}
Let $\tau = \inf\setsuch{t>0}{x_t\not\in\cD}$ denote the first-exit time of
$x_t$ from $\cD$. Then for any initial condition $x_0\in\cD$, we have 
\begin{equation}
 \label{ld08}
\lim_{\eps\to0} 2\eps\log\expecin{x_0}{\tau} = \inf_{z\in\partial\cD}
\Vbar(z)\bydef\Vbar\;. 
\end{equation}  
\end{theorem}

\begin{proof}
First one shows that for any
$x_0\in\cD$, it is likely to hit a small neighbourhood of $x^\star$ in finite
time. The large-deviation principle shows the existence of a time $T>0$,
independent of $\eps$, such that the probability of leaving $\cD$
in time $T$ is close to $p=\e^{-\Vbar/2\eps}$. Using the Markov
property to restart the process at multiples of $T$, one shows that the
number of time intervals of length $T$ needed to leave $\cD$ follows an
approximately geometric distribution, with expectation $1/p= \e^{\Vbar/2\eps}$
(these time intervals can be viewed as repeated \lq\lq attempts\rq\rq\ of the
process to leave $\cD$). The errors made in the different approximations vanish
when taking the limit~\eqref{ld08}.
\end{proof}

Wentzell and Freidlin also show that if the quasipotential reaches its minimum
on $\partial\cD$ at a unique, isolated point, then the first-exit location
$x_{\tau}$ concentrates in that point as $\eps\to0$. As for the
distribution of $\tau$, Day has shown that it is asymptotically exponential:

\begin{theorem}[\cite{Day1}]
\label{thm_Day} 
In the situation described above, 
\begin{equation}
 \label{ld09}
\lim_{\eps\to0} \bigprob{\tau > s \, \expec{\tau}} = \e^{-s} 
\end{equation} 
for all $s>0$. 
\end{theorem}

In general, the quasipotential $\Vbar$ has to be determined by minimising the
rate function~\eqref{ld03}, using either the Euler--Lagrange equations or the
associated Hamilton equations. In the gradient case $f=-\nabla V$, however, a
remarkable simplification occurs. Indeed, we can write 
\begin{align}
\nonumber
I(\varphi) 
&= \frac12 \int_0^T \norm{\dot\varphi(t) + \nabla V(\varphi(t))}^2 \6t \\
\nonumber
&= \frac12 \int_0^T \norm{\dot\varphi(t) - \nabla V(\varphi(t))}^2 \6t
+ 2 \int_0^T \pscal{\dot\varphi(t)}{\nabla V(\varphi(t))} \6t \\
&= \frac12 \int_0^T \norm{\dot\varphi(t) - \nabla V(\varphi(t))}^2 \6t
+ 2 \bigbrak{V(\varphi(T)) - V(\varphi(0))}\;.
 \label{ld10}
\end{align} 
The first term on the right-hand vanishes if $\varphi(t)$ is a solution of the
time-reversed deterministic system $\dot\varphi = + \nabla V(\varphi)$.
Connecting a local minimum $x^\star$ to a point in the basin of attraction of
$x^\star$ by such a solution is possible, if one allows for arbitrarily long 
time. Thus it follows that the quasipotential is given by 
\begin{equation}
 \label{ld11}
\Vbar = 2 \Bigbrak{\inf_{\partial\cD} V - V(x^\star)}\;.
\end{equation} 

\begin{cor}
\label{cor_wf}
In the double-well situation, 
\begin{equation}
 \label{ld12}
\lim_{\eps\to0} \eps\log \bigexpec{\tau_{\cB_\delta(y^\star)}} = V(z^\star) -
V(x^\star)\;. 
\end{equation}  
\end{cor}
\begin{proof}
Let $\cD$ be a set containing $x^\star$, and contained in the basin of
attraction of $x^\star$. One can choose $\cD$ in such a way that its boundary is
close to $z^\star$, and that the minimum of $V$ on $\partial\cD$ is attained
close to $z^\star$. Theorem~\ref{thm_FW} and~\eqref{ld11} show that a relation
similar to~\eqref{ld12} holds for the first-exit time from $\cD$. Then one shows
that once $x_t$ has left $\cD$, the average time needed to hit a small
neighbourhood of $y^\star$ is negligible compared to the expected first-exit
time from $\cD$. 
\end{proof}

\begin{remark}\hfill
\begin{enum}
\item	The case of more than two stable equilibrium points (or more general
attractors) can be treated by organising these points in a hierarchy of \lq\lq
cycles\rq\rq, which determines the exponent in Arrhenius' law and other
quantities of interest. See~\cite{FW,Freidlin1}. 

\item	As we have seen, the large-deviations approach is not limited to the
gradient case, but also allows to compute the exponent for irreversible
systems, by solving a variational problem. However, to our knowledge a rigorous
computation of the prefactor by this approach has not been achieved, as it
would require proving that the large-deviation functional $I$ also yields the
correct subexponential asymptotics. 

\item 	Sugiura~\cite{Sugiura95,Sugiura96a,Sugiura01} has built on these 
large-deviation results to derive approximations for the small eigenvalues of
the diffusion's generator (defined in the next section).
\end{enum}
\end{remark}


\newpage

\section{Analytic approaches and Kramers' law}
\label{sec_analytic}

The different analytic approaches to a proof of Kramers' law are based on
the fact that expected first-hitting times, when considered as a function of
the starting point, satisfy certain partial differential equations related to
Feynman--Kac formulas. 

To illustrate this fact, we consider the case of the symmetric simple random
walk on $\Z$. Fix two disjoint sets $A, B\subset \Z$, for instance of the form
$A=(-\infty,a]$ and $B=[b,\infty)$ with $a<b$ (\figref{fig_srw}). A first
quantity of interest is the probability of hitting $A$ before $B$, when starting
in a point $x$ between $A$ and $B$:
\begin{equation}
 \label{an01}
h_{A,B}(x) = \probin{x}{\tau_A < \tau_B}\;. 
\end{equation} 
For reasons that will become clear in Section~\ref{ssec_potential}, $h_{A,B}$
is called the \defwd{equilibrium potential}\/ between $A$ and $B$
(some authors call $h_{A,B}$ the \defwd{committor function}). 
Using the Markov property to restart the process after the first step, we can
write 
\begin{align}
\nonumber
h_{A,B}(x) ={}& \probin{x}{\tau_A < \tau_B, X_1=x+1} + 
\probin{x}{\tau_A < \tau_B, X_1=x-1} \\
\nonumber
={}& \pcondin{x}{\tau_A < \tau_B}{X_1=x+1}\probin{x}{X_1=x+1} \\
\nonumber
&{}+ \pcondin{x}{\tau_A < \tau_B}{X_1=x-1}\probin{x}{X_1=x-1} \\
={}& h_{A,B}(x+1) \cdot\tfrac12 + h_{A,B}(x-1) \cdot\tfrac12\;.
\label{an02} 
\end{align}
Taking into account the boundary conditions, we see that $h_{A,B}(x)$ satisfies
the linear Dirichlet boundary value problem  
\begin{align}
\nonumber
\Delta h_{A,B}(x) &= 0\;, & &x\in(A\cup B)^c\;, \\
\nonumber
h_{A,B}(x) &= 1\;, & &x\in A\;, \\
h_{A,B}(x) &= 0\;, & &x\in B\;,
\label{an03} 
\end{align}
where $\Delta$ denotes the discrete Laplacian 
\begin{equation}
 \label{an04}
(\Delta h)(x) = h(x-1) - 2h(x) + h(x+1)\;. 
\end{equation} 
A function $h$ satisfying $\Delta h=0$ is called a (discrete)
\defwd{harmonic}\/ function. In this one-dimensional situation, it is easy to
solve~\eqref{an03}: $h_{A,B}$ is simply a linear function of $x$ between $A$
and $B$. 

\begin{figure}
\centerline{\includegraphics*[clip=true,width=100mm]{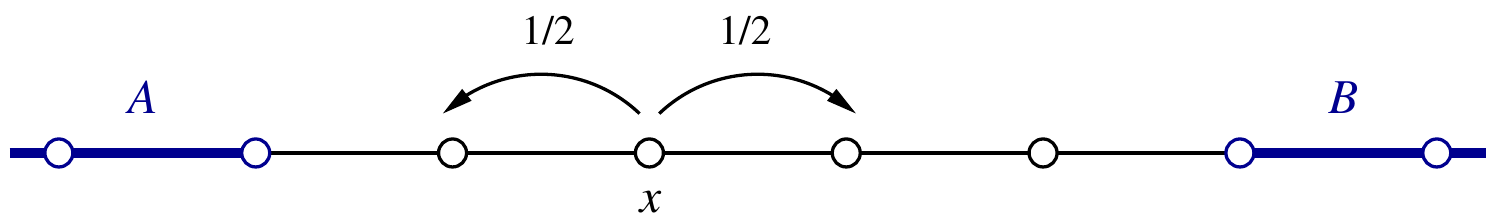}}
 \vspace{2mm}
\caption[]{Symmetric random walk on $\Z$ with two absorbing sets $A$, $B$. 
}
\label{fig_srw}
\end{figure}

A similar boundary value problem is satisfied by the mean first-hitting time of
$A$, $w_A(x)=\expecin{x}{\tau_A}$, assuming that $A$ is such
that the expectation exist (that is, the random walk on $A^c$ must be positive
recurrent). Here is an elementary computation (a shorter
derivation can be given using conditional expectations):
\begin{align}
\nonumber
w_A(x) &= \sum_k k\probin{x}{\tau_A=k} \\
\nonumber
&= \sum_k k \Bigbrak{\tfrac12 \probin{x-1}{\tau_A=k-1} + \tfrac12
\probin{x+1}{\tau_A=k-1}} \\
\nonumber
&= \sum_\ell (\ell+1)\Bigbrak{\tfrac12 \probin{x-1}{\tau_A=\ell} + \tfrac12
\probin{x+1}{\tau_A=\ell}} \\
&= \tfrac12 w_A(x-1) + \tfrac12 w_A(x+1) + 1\;.
\label{an05} 
\end{align}
In the last line we have used the fact that $\tau_A$ is almost surely finite, as
a consequence of positive recurrence. 
It follows that $w_A(x)$ satisfies the Poisson problem 
\begin{align}
\nonumber
\tfrac12\Delta w_A(x) &= -1\;, & &x\in A^c\;, \\
w_A(x) &= 0\;, & &x\in A\;. 
\label{an06} 
\end{align}
Similar relations can be written for more general quantities of the form 
$\expecin{x}{\e^{\lambda\tau_A}\indexfct{\tau_A<\tau_B}}$. 

In the case of Brownian motion on $\R^d$, the probability
$h_{A,B}(x)$ of hitting a set $A$ before another set $B$ satisfies the
Dirichlet problem 
\begin{align}
\nonumber
\tfrac12\Delta h_{A,B}(x) &= 0\;, & &x\in(A\cup B)^c\;, \\
\nonumber
h_{A,B}(x) &= 1\;, & &x\in A\;, \\
h_{A,B}(x) &= 0\;, & &x\in B\;,
\label{an07} 
\end{align}
where $\Delta$ now denotes the usual Laplacian in $\R^d$, and the expected
first-hitting time of $A$ satisfies the Poisson problem 
\begin{align}
\nonumber
\tfrac12\Delta w_A(x) &= -1\;, & &x\in A^c\;, \\
w_A(x) &= 0\;, & &x\in A\;.
\label{an08} 
\end{align}
For more general diffusions of the form 
\begin{equation}
 \label{an09}
\6x_t = -\nabla V(x_t)\6t + \sqrt{2\eps} \6W_t\;, 
\end{equation} 
Dynkin's formula~\cite{Dynkin65,Oeksendal} shows that 
similar relations as~\eqref{an07}, \eqref{an08} hold, with $\frac12\Delta$
replaced by the infinitesimal
generator of the diffusion, 
\begin{equation}
 \label{an10}
L = \eps\Delta - \nabla V(x)\cdot\nabla \;. 
\end{equation} 
Note that $L$ is the adjoint of the operator appearing in the Fokker--Planck
equation, which is more familiar to physicists. 
Thus by solving a boundary value problem involving a second-order differential
operator, one can in principle compute the expected first-hitting time, and
thus validate Kramers' law. This turns out to be possible in the
one-dimensional case, but no general solution exists in higher dimension, where 
one has to resort to perturbative techniques instead. 

\begin{remark}
\label{rem_recurrent}
Depending on the set $A$, Systems~\eqref{an06} and~\eqref{an07} need not
admit a bounded solution, owing to the fact that the symmetric random
walk and Brownian motion are null recurrent in dimensions $d=1,2$ and transient
in dimensions $d\geqs3$. A solution exists, however, for sets $A$ with bounded
complement. The situation is less restrictive for diffusions in a confining
potential $V$, which are usually positive recurrent. 
\end{remark}


\subsection{The one-dimensional case}
\label{ssec_1D}

In the case $d=1$, the generator of the diffusion has the form 
\begin{equation}
 \label{1D01}
(Lu)(x) = \eps u''(x) - V'(x)u'(x)\;, 
\end{equation} 
and the equations for $h_{A,B}(x)=\probin{x}{\tau_A<\tau_B}$ and 
$w_A(x) = \expecin{x}{\tau_A}$ can be solved explicitly. 

Consider the case where $A=(-\infty,a)$ and $B=(b,\infty)$ for some $a<b$, and
$x\in(a,b)$. Then it is easy to see that the equilibrium potential is given by 
\begin{equation}
 \label{1D02}
 h_{A,B}(x) = \frac{\displaystyle 
\int_x^b \e^{V(y)/\eps}\6y}{\displaystyle \int_a^b \e^{V(y)/\eps}\6y}\;.
\end{equation} 
Laplace's method to lowest order shows that for small $\eps$, 
\begin{equation}
 \label{1D03}
h_{A,B}(x) \simeq \exp \biggset{-\frac{1}{\eps} \biggbrak{\sup_{[a,b]}V -
\sup_{[x,b]}V}}\;. 
\end{equation} 
As one expects, the probability of hitting $A$ before $B$ is close to $1$ when
the starting point $x$ lies in the basin of attraction of $a$, and
exponentially small otherwise. 

\begin{figure}
\centerline{\includegraphics*[clip=true,height=50mm]{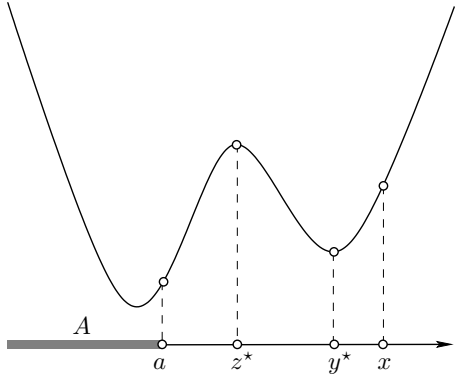}}
 \vspace{2mm}
\caption[]{Example of a one-dimensional potential for which Kramers'
law~\eqref{1D05} holds. 
}
\label{fig_pot1D}
\end{figure}

The expected first-hitting time of $A$ is given by the double integral 
\begin{equation}
 \label{1D04}
w_A(x) = \frac1\eps \int_a^x \int_z^\infty 
\e^{[V(z)-V(y)]/\eps}\6y\6z\;. 
\end{equation} 
If we assume that $x>y^\star>z^\star>a$, where $V$ has a local maximum in
$z^\star$ and a local minimum in $y^\star$ (\figref{fig_pot1D}), then the
integrand is maximal for $(y,z)=(y^\star,z^\star)$ and Laplace's method yields
exactly Kramers' law in the form 
\begin{equation}
 \label{1D05}
\expecin{x}{\tau_A} = w_A(x) = 
\frac{2\pi}{\sqrt{\abs{V''(z^\star)}V''(y^\star)}}
\e^{[V(z^\star)-V(y^\star)]/\eps}
\bigbrak{1+\Order{\sqrt{\eps}}}\;. 
\end{equation} 


\subsection{WKB theory}
\label{ssec_WKB}

The perturbative analysis of the infinitesimal generator~\eqref{an10} of the
diffusion in the limit $\eps\to0$ is strongly connected to semiclassical
analysis. Note that $L$ is not self-adjoint for the canonical scalar product,
but as a consequence of reversibility, it is in fact self-adjoint in
$L^2(\R^d,\e^{-V/\eps}\6x)$. This becomes immediately apparent when writing $L$
in the equivalent form 
\begin{equation}
 \label{WKB01}
L = \eps \e^{V/\eps} \nabla \cdot \e^{-V/\eps}\nabla 
\end{equation} 
(just write out the weighted scalar product). It follows that the conjugated
operator 
\begin{equation}
 \label{WKB02}
\widetilde L = \e^{-V/2\eps}L \e^{V/2\eps} 
\end{equation} 
is self-adjoint in $L^2(\R^d, \6x)$. In fact, a simple computation shows that
$\widetilde L$ is a Schr\"odinger operator of the form 
\begin{equation}
 \label{WKB03}
\widetilde L = \eps\Delta + \frac{1}{\eps} U(x)\;, 
\end{equation} 
where the potential $U$ is given by 
\begin{equation}
 \label{WKB04}
U(x) = \frac12 \eps \Delta V(x) - \frac14 \norm{\nabla V(x)}^2\;. 
\end{equation} 

\begin{example}
\label{ex_WKB}
For a double-well potential of the form 
\begin{equation}
 \label{WKB05}
V(x) = \frac14 x^4 - \frac12 x^2\;,  
\end{equation}  
the potential $U$ in the Schr\"odinger operator takes the form 
\begin{equation}
 \label{WKB06}
U(x) = -\frac14 x^2 (x^2-1)^2 + \frac12 \eps (x^2-1)^2\;. 
\end{equation} 
Note that this potential has $3$ local minima at almost the same height, namely
two of them at $\pm1$ where $U(\pm1)=0$ and one at $0$ where $U(0)=\eps/2$. 
\end{example}

One may try to solve the Poisson problem $Lw_A=-1$ by WKB-techniques in order
to validate Kramers' formula. A closely related problem is to determine the
spectrum of $L$. Indeed, it is known that if the potential $V$ has $n$ local
minima, then $L$ admits $n$ exponentially small eigenvalues, which are related
to the inverse of expected transition times between certain potential minima.
The associated eigenfunctions are concentrated in potential wells and represent
metastable states.

The WKB-approach has been investigated, e.g., in
\cite{SchussMatkowsky1979,BuslovMakarov1988,KolokoltsovMakarov1996,MS1}.
See~\cite{Kolokoltsov00} for a recent review. A mathematical justification of
this formal procedure is often possible, using hard analytical methods such as
microlocal analysis~\cite{HelfferSjostrand1,HelfferSjostrand2,HelfferSjostrand3,
HelfferSjostrand4}, which have been developed for quantum tunnelling problems.
The difficulty in the case of Kramers' law is that due to the form~\eqref{WKB04}
of the Schr\"odinger potential $U$, a phenomenon called \lq\lq tunnelling
through nonresonant wells\rq\rq\ prevents the existence of a single WKB ansatz,
valid in all $\R^d$. One thus has to use different ansatzes in different regions
of space, whose asymptotic expansions have to be matched at the boundaries, a
procedure that is difficult to justify mathematically. 

Rigorous results on the eigenvalues of $L$ have nevertheless been obtained with
different methods in \cite{HolleyKusuokaStroock1989,Miclo1995,Mathieu1995}, but
without a sufficiently precise control of their subexponential behaviour as
would be required to rigorously prove Kramers' law. 


\subsection{Potential theory}
\label{ssec_potential}

\begin{figure}
\centerline{\includegraphics*[clip=true,height=50mm]{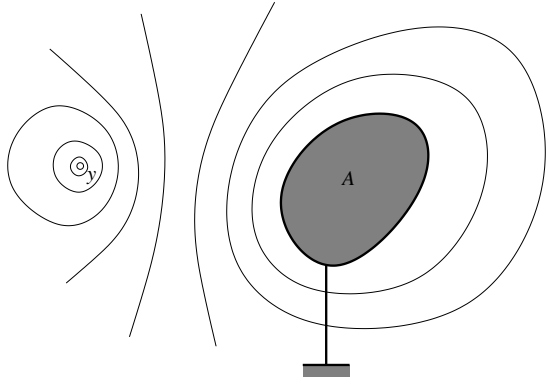}}
 \vspace{2mm}
\caption[]{Green's function $G_{A^c}(x,y)$ for 
Brownian motion is equal to the electrostatic potential in $x$ created by a
unit charge in $y$ and a grounded conductor in $A$. 
}
\label{fig_Green}
\end{figure}

Techniques from potential theory have been widely used in probability
theory~\cite{Kakutani1945,Doob1984,Doyle_Snell,Sznitman_book_1998}.
Although Wentzell may have had in mind its application to Kramers'
law~\cite{Wentzell1973}, this program has been systematically carried out only
quite recently by Bovier, Eckhoff, Gayrard and Klein~\cite{BEGK,BGK}. 

We will explain the basic idea of this approach in the simple setting of
Brownian motion in $\R^d$, which is equivalent to an electrostatics problem.
Recall that the first-hitting time $\tau_A$ of a set $A\subset\R^d$ satisfies
the Poisson problem~\eqref{an06}. It can thus be expressed as 
\begin{equation}
 \label{pot01}
w_A(x) = -\int_{A^c} G_{A^c}(x,y)\6y\;, 
\end{equation} 
where $G_{A^c}(x,y)$ denotes Green's function, which is the formal solution of 
\begin{align}
\nonumber
\tfrac12\Delta u(x) &= \delta(x-y)\;, & &x\in A^c\;, \\
u(x) &= 0\;, & &x\in A\;.
\label{pot02} 
\end{align}
Note that in electrostatics, $G_{A^c}(x,y)$ represents the value at $x$ of the
electric potential created by a unit point charge at $y$, when the set $A$ is
occupied by a grounded conductor (\figref{fig_Green}).  

Similarly, the solution $h_{A,B}(x)=\probin{x}{\tau_A<\tau_B}$ of the Dirichlet
problem~\eqref{an07} represents the electric potential at $x$, created by a
capacitor formed by two conductors at $A$ and $B$, at respective electric
potential $1$ and $0$ (\figref{fig_capacitor}). Hence the name
\defwd{equilibrium potential}. If $\rho_{A,B}$ denotes the surface charge
density on the two conductors, the potential can thus be expressed in the form 
\begin{equation}
 \label{pot03}
h_{A,B}(x) = \int_{\partial A} G_{B^c}(x,y) \rho_{A,B}(\6y)\;.
\end{equation} 
Note finally that the capacitor's capacity is simply equal to the total
charge on either of the two conductors, given by 
\begin{equation}
 \label{pot04}
\capacity_A(B) = \biggabs{\int_{\partial A} \rho_{A,B}(\6y)}\;. 
\end{equation} 

\begin{figure}
\centerline{\includegraphics*[clip=true,height=50mm]{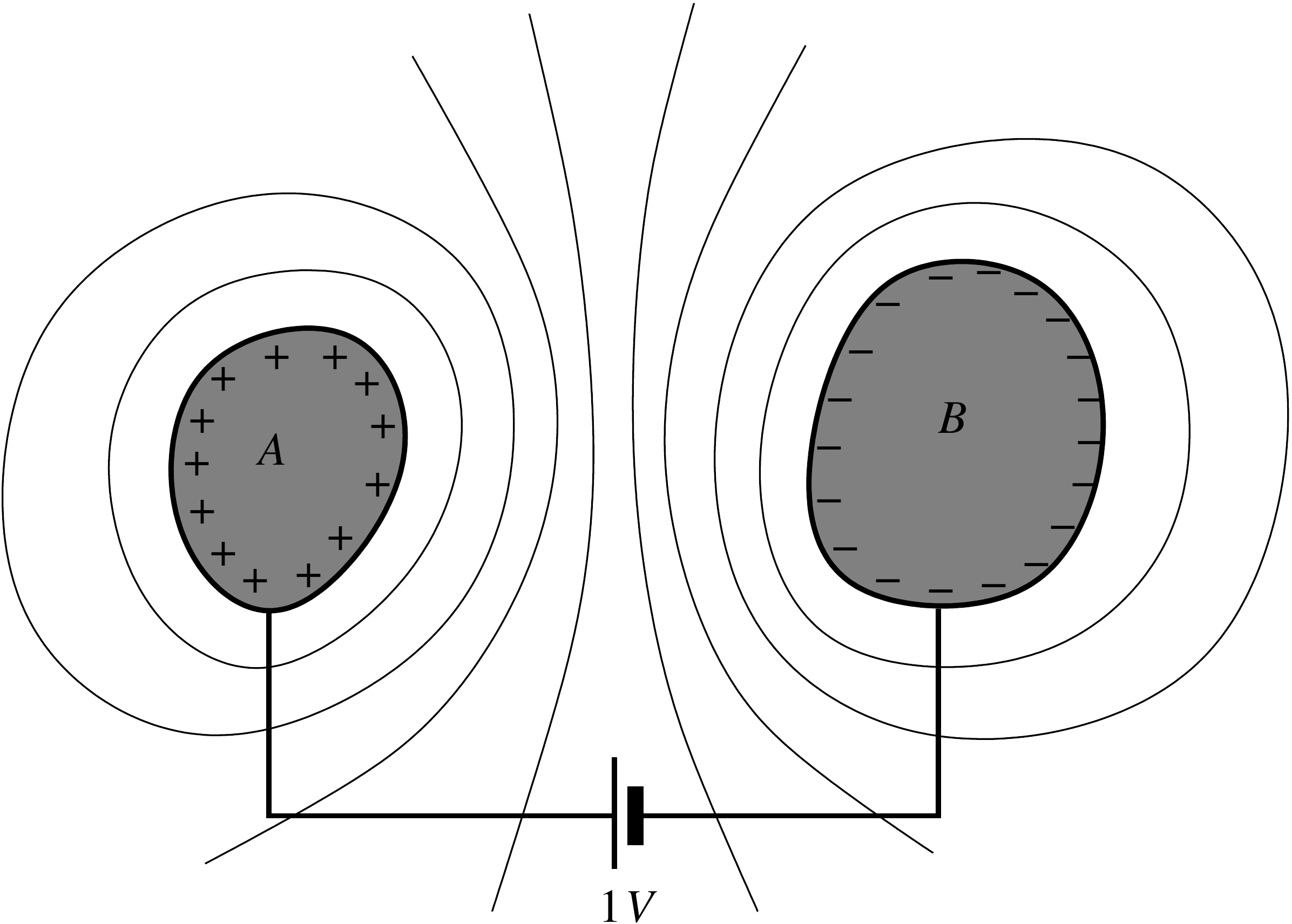}}
 \vspace{2mm}
\caption[]{The function $h_{A,B}(x)=\probin{x}{\tau_A<\tau_B}$ is equal to the
electric potential in $x$ of a capacitor with conductors in $A$ and $B$, at
respective potential $1$ and $0$. 
}
\label{fig_capacitor}
\end{figure}

The key observation is that even though we know neither Green's function, nor
the surface charge density, the expressions~\eqref{pot01},~\eqref{pot03}
and~\eqref{pot04} can be combined to yield a useful relation between expected
first-hitting time and capacity. Indeed, let $C$ be a small ball centred in
$x$. Then we have 
\begin{align}
\nonumber
\int_{A^c} h_{C,A}(y)\6y 
&= \int_{A^c} \int_{\partial C} G_{A^c}(y,z) \rho_{C,A}(\6z) \6y \\
&= -\int_{\partial C} w_A(z) \rho_{C,A}(\6z)\;.
\label{pot05} 
\end{align}
We have used the symmetry $G_{A^c}(y,z)=G_{A^c}(z,y)$, which is a consequence
of reversibility. Now since $C$ is a small ball, if $w_A$ does not vary too
much in $C$, the last term in~\eqref{pot05} will be close to
$w_A(x)\capacity_C(A)$. This can be justified by using a Harnack inequality,
which provides bounds on the oscillatory part of harmonic functions. As a
result, we obtain the estimate 
\begin{equation}
 \label{pot06}
\bigexpecin{x}{\tau_A} = w_A(x) 
\simeq \frac{\displaystyle \int_{A^c} h_{C,A}(y)\6y}{\capacity_C(A)}\;. 
\end{equation} 
This relation is useful because capacities can be estimated by a variational
principle. Indeed, using again the electrostatics analogy, for unit potential
difference, the capacity is equal to the capacitor's electrostatic energy,
which is equal to the total energy of the electric field $\nabla h$:
\begin{equation}
 \label{pot07}
\capacity_A(B) = \int_{(A\cup B)^c} \norm{\nabla h_{A,B}(x)}^2 \6x\;. 
\end{equation} 
In potential theory, this integral is known as a~\defwd{Dirichlet form}. A
remarkable fact is that the capacitor at equilibrium minimises the
electrostatic energy, namely, 
\begin{equation}
 \label{pot08}
\capacity_A(B) = \inf_{h\in\cH_{A,B}}
\int_{(A\cup B)^c} \norm{\nabla h(x)}^2 \6x\;,
\end{equation} 
where $\cH_{A,B}$ denotes the set of all sufficiently regular functions $h$
satisfying the boundary conditions in~\eqref{an07}. 
Similar considerations can be made in the case of general reversible diffusions
of the form 
\begin{equation}
 \label{pot09}
\6x_t = -\nabla V(x_t)\6t + \sqrt{2\eps} \6W_t\;, 
\end{equation} 
a crucial point being that reversibility implies the symmetry 
\begin{equation}
 \label{pot11}
\e^{-V(x)/\eps} G_{A^c}(x,y) =  \e^{-V(y)/\eps} G_{A^c}(y,x)\;.
\end{equation} 
This allows to obtain the estimate 
\begin{equation}
 \label{pot12}
 \bigexpecin{x}{\tau_A} = w_A(x) 
\simeq \frac{\displaystyle \int_{A^c} h_{C,A}(y)\e^{-V(y)/\eps}\6y}
{\capacity_C(A)}\;,
\end{equation} 
where the capacity is now defined as 
\begin{equation}
 \label{pot13}
\capacity_A(B) = \inf_{h\in\cH_{A,B}}
\int_{(A\cup B)^c} \norm{\nabla h(x)}^2 \e^{-V(x)/\eps}\6x\;.
\end{equation} 
The numerator in~\eqref{pot12} can be controlled quite easily. In fact, rather
rough a priori bounds suffice to show that if $x^\star$ is a potential minimum,
then $h_{C,A}$ is exponentially close to $1$ in the basin of attraction of
$x^\star$. Thus by straightforward Laplace asymptotics, we obtain 
\begin{equation}
 \label{pot14}
 \int_{A^c} h_{C,A}(y)\e^{-V(y)/\eps}\6y
= \frac{(2\pi\eps)^{d/2}\e^{-V(x^\star)/\eps}}{\sqrt{\det(\hessian
V(x^\star))}} \bigbrak{1+\Order{\sqrt{\eps}\abs{\log\eps}}}\;.
\end{equation}  
Note that this already provides one \lq\lq half\rq\rq\ of Kramers'
law~\eqref{in09}. The other half thus has to come from the capacity
$\capacity_C(A)$, which can be estimated with the help of the variational
principle~\eqref{pot13}. 

\begin{theorem}[\cite{BEGK}]
\label{thm_BEGK} 
In the double-well situation, Kramers' law holds in the sense that 
\begin{equation}
 \label{pot15}
\bigexpecin{x}{\tau_{\cB_\eps(y^\star)}} = 
\frac{2\pi}{\abs{\lambda_1(z^\star)}} 
\sqrt{\frac{\abs{\det(\hessian V(z^\star))}}{\det(\hessian V(x^\star))}} 
\e^{[V(z^\star)-V(x^\star)]/\eps}
\bigbrak{1+\Order{\eps^{1/2}\abs{\log\eps}^{3/2}}}\;,
\end{equation} 
where $\cB_\eps(y^\star)$ is the ball of radius $\eps$ (the same $\eps$ as in
the diffusion coefficient) centred in~$y^\star$. 
\end{theorem}
\begin{proof}
In view of~\eqref{pot12} and~\eqref{pot14}, it is sufficient to obtain
sharp upper and lower bounds on the capacity, of the form
\begin{equation}
 \label{pot16}
\capacity_C(A) = \frac1{2\pi} \sqrt{\frac{(2\pi\eps)^d\abs{\lambda_1(z)}}
{\lambda_2(z)\dots\lambda_d(z)}} \e^{-V(z)/\eps}
  \bigbrak{1+\Order{\eps^{1/2}\abs{\log\eps}^{3/2}}}\;.
\end{equation} 
The variational principle~\eqref{pot13} shows that the Dirichlet form of any
function $h\in\cH_{A,B}$ provides an upper bound on the capacity. It is thus
sufficient to construct an appropriate $h$. It turns out that taking
$h(x)=h_1(x_1)$, depending only on the projection $x_1$ of $x$ on the unstable
manifold of the saddle, with $h_1$ given by the solution~\eqref{1D02} of the
one-dimensional case, does the job. 

The lower bound is a bit more tricky to obtain. Observe first that restricting
the domain of integration in the Dirichlet form~\eqref{pot13} to a small
rectangular box centred in the saddle decreases the value of the integral.
Furthermore, the integrand $\norm{\nabla h(x)}^2$ is bounded below by the
derivative in the unstable direction squared. For given values of the
equilibrium potential $h_{A,B}$ on the sides of the box intersecting the
unstable manifold of the saddle, the Dirichlet form can thus be bounded below by
solving a one-dimensional variational problem. Then rough a priori bounds on the
boundary values of $h_{A,B}$ yield the result. 
\end{proof}

\begin{remark}
For simplicity, we have only presented the result on the expected transition
time for the double-well situation. Results in~\cite{BEGK,BGK} also include the
following points:
\begin{enum}
\item	The distribution of $\tau_{\cB_\eps(y)}$ is asymptotically exponential,
in the sense of~\eqref{ld09}.
\item	In the case of more than $2$ local minima, Kramers' law holds for
transitions between local minima provided they are appropriately ordered. See
Example~\ref{ex_multiwell} below. 
\item	The small eigenvalues of the generator $L$ can be sharply estimated,
the leading terms being equal to inverses of mean transition times.
\item	The associated eigenfunctions of $L$ are well-approximated by
equilibrium potentials $h_{A,B}$ for certain sets $A, B$. 
\end{enum}
\end{remark}

If the potential $V$ has $n$ local minima, there exists an ordering 
\begin{equation}
 \label{pot17}
x^\star_1 \prec x^\star_2\prec\dots\prec x^\star_n 
\end{equation} 
such that Kramers' law holds for the transition time from each $x^\star_{k+1}$
to the set $\cM_k=\set{x^\star_1,\dots,x^\star_k}$. The ordering is defined in
terms of communication heights by the condition 
\begin{equation}
 \label{pot18}
H(x^\star_k,\cM_{k-1}) \leqs \min_{i<k} H(x^\star_i,\cM_k\setminus x^\star_i) -
\theta 
\end{equation} 
for some $\theta>0$. This means that the minima are ordered from deepest to
shallowest. 

\begin{figure}
\centerline{\includegraphics*[clip=true,height=50mm]{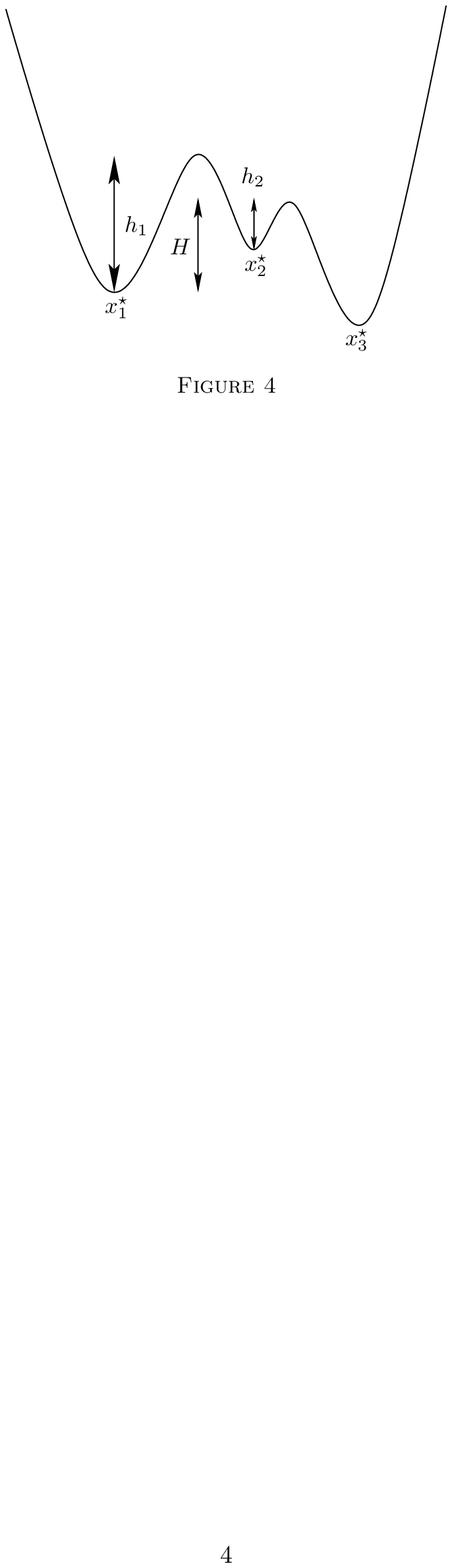}}
 \vspace{2mm}
\caption[]{Example of a three-well potential, with associated metastable
hierarchy. The relevant communication heights are given by
$H(x^\star_2,\set{x^\star_1,x^\star_3})=h_2$ and $H(x^\star_1,x^\star_3)=h_1$.
}
\label{fig_multiwell}
\end{figure}

\begin{example}
\label{ex_multiwell} 
Consider the three-well potential shown in~\figref{fig_multiwell}. The
metastable ordering is given by 
\begin{equation}
 \label{pot20}
x^\star_3 \prec x^\star_1 \prec x^\star_2\;, 
\end{equation} 
and Kramers' law holds in the form 
\begin{equation}
 \label{pot21}
\bigexpecin{x^\star_1}{\tau_3} \simeq C_1 \e^{h_1/\eps}\;, 
\qquad
 \bigexpecin{x^\star_2}{\tau_{\set{1,3}}} \simeq C_2 \e^{h_2/\eps}\;,
\end{equation} 
where the constants $C_1,C_2$ depend on second derivatives of $V$. However, it
is \emph{not}\/ true that $\expecin{x^\star_2}{\tau_3} \simeq C_2
\e^{h_2/\eps}$. In fact, $\expecin{x^\star_2}{\tau_3}$ is rather of the order
$\e^{H/\eps}$. This is due to the fact that even though when starting in
$x^\star_2$, the process is very unlikely to hit $x^\star_1$ before $x^\star_3$
(this happens with a probability of order $\e^{-(h_1-H)/\eps}$), this is
overcompensated by the very long waiting time in the well $x^\star_1$ (of order
$\e^{h_1/\eps}$) in case this happens. 
\end{example}


\subsection{Witten Laplacian}
\label{ssec_Witten}

In this section, we give a brief account of another successful approach to
proving Kramers' law, based on WKB theory for the Witten Laplacian. It provides
a good example of the fact that problems may be made more accessible to
analysis by generalising them. 

Given a compact, $d$-dimensional, orientable manifold $M$, equipped with a
smooth metric $g$, let $\Omega^p(M)$ be the set of differential forms of order
$p$ on $M$. The exterior derivative $\dd$ maps a $p$-form to a $(p+1)$-form. We
write $\dd^{(p)}$ for the restriction of $\dd$ to $\Omega^p(M)$. The sequence 
\begin{equation}
 \label{Witten01}
0 \rightarrow \Omega^0(M) \xrightarrow{\dd^{(0)}} \Omega^1(M) 
\xrightarrow{\dd^{(1)}} \dots 
\xrightarrow{\dd^{(d-1)}} \Omega^d (M) \xrightarrow{\dd^{(d)}} 0
\end{equation} 
is called the \defwd{de Rham complex}\/ associated with $M$. 

Differential forms in the image $\image \dd^{(p-1)}$ are called \defwd{exact},
while differential forms in the kernel $\ker \dd^{(p)}$ are called
\defwd{closed}. Exact forms are closed, that is, $\dd^{(p)}\circ\dd^{(p-1)}=0$
or in short $\dd^2=0$. However, closed forms are not necessarily exact.
Hence the idea of considering equivalence classes of differential forms
differing by an exact form. The vector spaces
\begin{equation}
 \label{Witten02}
H^p(M) = \frac{\ker \dd^{(p)}}{\image \dd^{(p-1)}} 
\end{equation}  
are thus not necessarily trivial, and contain information on the global topology
of $M$. They form the so-called \defwd{de Rham cohomology}. 

The metric $g$ induces a natural scalar product $\pscal{\cdot}{\cdot}_p$ on
$\Omega^p(M)$ (based on the Hodge isomorphism $*$). The \defwd{codifferential}\/
on $M$ is the formal adjoint $\dd^*$ of $\dd$, which maps $(p+1)$-forms to
$p$-forms and satisfies 
\begin{equation}
 \label{Witten03}
\pscal{\dd \omega}{\eta}_{p+1} =  \pscal{\omega}{\dd^*\eta}_p 
\end{equation} 
for all $\omega\in\Omega^p(M)$ and $\eta\in\Omega^{p+1}(M)$. 
The \defwd{Hodge Laplacian}\/ is defined as the symmetric non-negative operator 
\begin{equation}
 \label{Witten04}
\Delta_H = \dd\dd^* + \dd^*\dd  = (\dd+\dd^*)^2\;,
\end{equation} 
and we write $\Delta^{(p)}_H$ for its
restriction to $\Omega^p$. In the Euclidean case $M=\R^d$, using integration by
parts in~\eqref{Witten03} shows that  
\begin{equation}
 \label{Witten05}
 \Delta^{(0)}_H = - \Delta\;,
\end{equation} 
where $\Delta$ is the usual Laplacian. 
Differential forms $\gamma$ in the kernel $\cH^p_\Delta(M) = \ker\Delta^{(p)}_H$
are called \defwd{$p$-harmonic forms}. They are both closed ($\dd\gamma=0$) and
co-closed ($\dd^*\gamma=0$). Hodge has shown (see,
e.g.~\cite{GriffithsHarris1994}) that any differential form
$\omega\in\Omega^p(M)$ admits a unique decomposition 
\begin{equation}
 \label{Witten05b}
\omega = \dd\alpha + \dd^*\beta + \gamma\;, 
\end{equation}  
where $\gamma$ is $p$-harmonic. As a consequence, $\cH^p_\Delta(M)$ is
isomorphic to the $p$th de Rham cohomology group $H^p(M)$.

Given a potential $V:M\to\R$, the \defwd{Witten Laplacian}\/ is defined in a
similar way as the Hodge Laplacian by 
\begin{equation}
\label{Witten06}
\Delta_{V,\eps} = \dd_{V,\eps} \dd^*_{V,\eps} + \dd^*_{V,\eps} \dd_{V,\eps}\;,
\end{equation} 
where $\dd_{V,\eps}$ denotes the deformed exterior derivative 
\begin{equation}
\label{Witten07}
\dd_{V,\eps} = \eps \e^{-V/2\eps} \dd \e^{V/2\eps}\;.
\end{equation} 
As before, we write $\smash{\Delta^{(p)}_{V,\eps}}$ for the restriction of
$\Delta_{V,\eps}$ to $\Omega^p(M)$. A direct computation shows that in the
Euclidean case $M=\R^d$, 
\begin{equation}
 \label{Witten08}
\Delta^{(0)}_{V,\eps} = -\eps^2\Delta + \frac14\norm{\nabla V}^2 -
\frac12\eps\Delta V\;,
\end{equation} 
which is equivalent, up to a scaling, to the Schr\"odinger
operator~\eqref{WKB03}. 

The interest of this approach lies in the fact that while eigenfunctions of
$\smash{\Delta^{(0)}_{V,\eps}}$ are concentrated near local minima of the
potential $V$, those of $\smash{\Delta^{(p)}_{V,\eps}}$ for $p\geqs1$ are
concentrated near saddles of index $p$ of $V$. This makes them easier to
approximate by WKB theory. The intertwining relations 
\begin{equation}
 \label{Witten09}
 \Delta^{(p+1)}_{V,\eps} \dd^{(p)}_{V,\eps}
= \dd^{(p)}_{V,\eps} \Delta^{(p)}_{V,\eps}\;,
\end{equation} 
which follow from $\dd^2=0$, then allow to infer more precise information on
the spectrum of $\smash{\Delta^{(0)}_{V,\eps}}$, and hence of the generator $L$
of the diffusion~\cite{HelfferNier05}. 

This approach has been used by Helffer, Klein and Nier~\cite{HelfferKleinNier04}
to prove Kramers' law~\eqref{in09} with a full asymptotic expansion of the
prefactor $C=C(\eps)$, in~\cite{HelfferNier06} to describe the case of general
manifolds with boundary, and by Le Peutrec~\cite{LePeutrec10} for the case
with Neumann boundary conditions. General expressions for the small eigenvalues
of all $p$-Laplacians have been recently derived
in~\cite{LePeutrec11,LePeutrecNierViterbo12}.


\section{Generalisations and limits}
\label{sec_gen}

In this section, we discuss two generalisations of Kramers' formula, and one
irreversible case, where Arrhenius' law still holds true, but the prefactor is
no longer given by Kramers' law.


\subsection{Non-quadratic saddles}
\label{ssec_nonquad}

Up to now, we have assumed that all critical points are quadratic saddles, that
is, with a nonsingular Hessian. Although this is true generically, as soon as
one considers potentials depending on one or several parameters, degenerate
saddles are bound to occur. See for instance~\cite{BFG06a,BFG06b} for a natural
system displaying many bifurcations involving nonquadratic saddles. Obviously,
Kramers' law~\eqref{in09} cannot be true in the presence of singular Hessians,
since it would predict either a vanishing or an infinite prefactor. In fact, in
such cases the prefactor will depend on higher-order terms of the Taylor
expansion of the potential at the relevant critical points~\cite{Stein04}. The
main problem is thus to determine the prefactor's leading term. 

There are two (non-exclusive) cases to be considered: the starting
potential minimum $x^\star$ or the relevant saddle $z^\star$ is non-quadratic.
The potential-theoretic approach presented in Section~\ref{ssec_potential}
provides a simple way to deal with both cases. In the first case, it is in fact
sufficient to carry out Laplace's method for~\eqref{pot14} when the potential
$V$ has a nonquadratic minimum in $x^\star$, which is straightforward. 

We discuss the more interesting case of the saddle $z^\star$ being
non-quadratic. A general classification of non-quadratic saddles, based on
normal-form theory, is given in~\cite{BG2010}. 

Consider the case where in appropriate coordinates, the potential near the
saddle admits an expansion of the form 
\begin{equation}
 \label{nq01}
V(y) = -u_1(y_1) + u_2(y_2,\dots,y_k) + \frac{1}{2}\sum_{j=k+1}^d \lambda_jy_j^2
+ \Order{\norm{y}^{r+1}}\;, 
\end{equation} 
for some $r\geqs 2$ and $2\leqs k\leqs d$. The functions $u_1$ and $u_2$ may
take negative values in a small neighbourhood of the origin, of the order of
some power of $\eps$, but should become positive and grow outside this
neighbourhood. In that case, we have the following estimate of the capacity:

\begin{theorem}[\cite{BG2010}]
There exists an explicit $\beta>0$, depending on the growth of $u_1$ and
$u_2$, such that in the double-well situation the capacity is given by 
\begin{equation}
 \label{nq02}
\eps \frac{\displaystyle
\int_{\R^{k-1}} \e^{-u_2(y_2,\dots,y_k)/\eps} \6y_2\dots\6y_k}
{\displaystyle\int_{-\infty}^\infty \e^{-u_1(y_1)/\eps}\6y_1}
\prod_{j=k+1}^d \sqrt{\frac{2\pi\eps}{\lambda_j}}
\Bigbrak{1+\Order{\eps^\beta\abs{\log\eps}^{1+\beta}}}\;. 
\end{equation}
\end{theorem}

We discuss one particular example, involving a pitchfork bifurcation.
See~\cite{BG2010} for more examples. 
 
\begin{example}
\label{ex_pitchfork}
Consider the case $k=2$ with 
\begin{align}
\nonumber
u_1(y_1) &= -\frac12\abs{\lambda_1}y_1^2\;, \\
u_2(y_2) &= \frac12\lambda_2 y_2^2 + C_4 y_2^4\;,
 \label{nq03}
\end{align}  
where $\lambda_1<0$ and $C_4>0$ are bounded away from $0$. We assume that the
potential is even in~$y_2$. For $\lambda_2>0$,
the origin is an isolated quadratic saddle. At $\lambda_2=0$, the origin
undergoes a pitchfork bifurcation, and for $\lambda_2<0$, there are two saddles
at $y_2=\pm\sqrt{\abs{\lambda_2}/4C_4}+\Order{\lambda_2}$. Let $\mu_1, \dots,
\mu_d$ denote the eigenvalues of the Hessian of $V$ at these saddles. 

The integrals in~\eqref{nq02} can be computed explicitly, and yield the
following prefactors in Kramers' law: 
\begin{itemiz}
\item	For $\lambda_2\geqs0$, the prefactor is given by 
\begin{equation}
 \label{nq04}
C(\eps) = 2\pi \sqrt{\frac{(\lambda_2+\sqrt{2\eps C_4}\,)
\lambda_3\dots\lambda_d}{\abs{\lambda_1}\det(\hessian V(x^\star))}}
\frac1{\Psi_+(\lambda_2/\sqrt{2\eps C_4})}\;,
\end{equation} 
where the function $\Psi_+$ is bounded above and below by positive constants,
and is given in terms of the modified Bessel function of the second kind
$K_{1/4}$ by 
\begin{equation}
 \label{nq05}
\Psi_+(\alpha) = \sqrt{\frac{\alpha(1+\alpha)}{8\pi}} \e^{\alpha^2/16} 
K_{1/4} \biggpar{\frac{\alpha^2}{16}}\;.
\end{equation} 

\item	For $\lambda_2<0$, the prefactor is given by 
\begin{equation}
 \label{nq06}
C(\eps) = 2\pi \sqrt{\frac{(\mu_2+\sqrt{2\eps C_4}\,)
\mu_3\dots\mu_d}{\abs{\mu_1}\det(\hessian V(x^\star))}}
\frac1{\Psi_-(\mu_2/\sqrt{2\eps C_4})}\;,
\end{equation} 
where the function $\Psi_-$ is again bounded above and below by positive
constants, and given in terms of the modified Bessel functions of the first
kind $I_{\pm 1/4}$ by 
\begin{equation}
 \label{nq07}
\Psi_-(\alpha) = \sqrt{\frac{\pi\alpha(1+\alpha)}{32}} \e^{-\alpha^2/64} 
\biggbrak{I_{-1/4} \biggpar{\frac{\alpha^2}{64}}
+ I_{1/4} \biggpar{\frac{\alpha^2}{64}}}\;.
\end{equation}
\end{itemiz}

\begin{figure}
\centerline{\includegraphics*[clip=true,height=50mm]{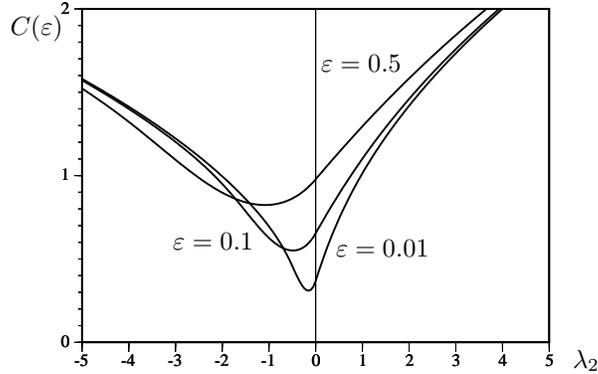}}
 \vspace{2mm}
\caption[]{The prefactor $C(\eps)$ in Kramers' law when the potential undergoes
a pitchfork bifurcation as the parameter $\lambda_2$ changes sign. The minimal
value of $C(\eps)$ has order $\eps^{1/4}$.  
}
\label{fig_Psi}
\end{figure}

As long as $\lambda_2$ is bounded away from $0$, we recover the usual Kramers
prefactor. When $\abs{\lambda_2}$ is smaller than $\sqrt{\eps}$, however, the
term $\sqrt{2\eps C_4}$ dominates, and yields a prefactor of order $\eps^{1/4}$
(see~\figref{fig_Psi}). The exponent $1/4$ is characteristic of this particular
type of bifurcation. 

The functions $\Psi_\pm$ determine a multiplicative constant, which is close to
$1$ when $\lambda_2\gg\sqrt{\eps}$, to $2$ when $\lambda_2\ll - \sqrt{\eps}$,
and to $\Gamma(1/4)/(2^{5/4}\sqrt{\pi})$ for $\abs{\lambda_2}\ll\sqrt{\eps}$.
The factor $2$ for large negative $\lambda_2$ is due to the presence of two
saddles.
\end{example}


\subsection{SPDEs}
\label{ssec_SPDE}

Metastability can also be displayed by parabolic stochastic partial differential
equations of the form 
\begin{equation}
 \label{SPDE01}
\partial_t u(t,x) = \partial_{xx} u(t,x) + f(u(t,x)) + \sqrt{2\eps}
\ddot{W}_{tx}\;, 
\end{equation} 
where $\ddot{W}_{tx}$ denotes space-time white noise (see,
e.g.~\cite{Walsh_SPDEs}). We consider here the simplest case where $u(t,x)$
takes values in $\R$, and $x$ belongs to an interval $[0,L]$, with either
periodic or Neumann boundary conditions (b.c.). Equation~\eqref{SPDE01} can be
considered as an infinite-dimensional gradient system, with potential 
\begin{equation}
 \label{SPDE02}
V[u] = \int_0^L \biggbrak{\frac12 u'(x)^2 + U(u(x))}\6x\;, 
\end{equation} 
where $U'(x)=-f(x)$. Indeed, using integration by parts one obtains that the
Fr\'echet derivative of $V$ in the direction $v$ is given by 
\begin{equation}
 \label{SPDE02A}
\dtot{}{\eta} {V[u+\eta v]} \Bigevalat{\eta=0} 
= -\int_0^L \bigbrak{u''(x)+f(u(x))} v(x) \6x\;,
\end{equation} 
which vanishes on stationary solutions of the deterministic system 
$\partial_t u = \partial_{xx} u + f(u)$.

In the case of the double-well potential $U(u)=\frac14 u^4 - \frac12 u^2$, the
equivalent of Arrhenius' law has been proved by Faris and
Jona-Lasinio~\cite{Faris_JonaLasinio82}, based on a large-deviation principle. 
For both periodic and Neumann b.c., $V$
admits two global minima $u_\pm(x)\equiv \pm1$. The relevant saddle between
these solutions depends on the value of $L$. For Neumann b.c.,
it is given by 
\begin{equation}
 \label{SPDE03}
u_0(x) = 
\begin{cases}
0 & \text{if $L\leqs\pi$\;,} \\
\pm \sqrt{\frac{2m}{m+1}} \sn \Bigpar{\frac{x}{\sqrt{m+1}}+\JK(m),m}
& \text{if $L>\pi$\;,}
\end{cases} 
\end{equation} 
where $2\sqrt{m+1}\JK(m)=L$, $\JK$ denotes the elliptic integral of the first
kind, and $\sn$ denotes Jacobi's elliptic sine. There is a pitchfork
bifurcation at $L=\pi$. The exponent in Arrhenius' law
is given by the difference $V[u_0]-V[u_-]$, which can be computed explicitly in
terms of elliptic integrals. 

The prefactor in Kramers' law has been computed by Maier and Stein, for
various b.c., and $L$ bounded away from the bifurcation value
($L=\pi$ for Neumann and Dirichlet b.c., $L=2\pi$ for periodic
b.c.)~\cite{Maier_Stein_PRL_01,Maier_Stein_SPIE_2003,Stein_JSP_04}. 
The basic observation is that the second-order Fr\'echet derivative of $V$ at a
stationary solution $u$ is the quadratic form 
\begin{equation}
 \label{SPDE04A}
(v_1,v_2) \mapsto \pscal{v_1}{Q[u]v_2}\;, 
\end{equation} 
where
\begin{equation}
 \label{SPDE04B}
Q[u]v(x) = -v''(x) - f'(u(x))v(x)\;.  
\end{equation} 
Thus the r\^ole of the eigenvalues of the Hessian is played by the eigenvalues
of the second-order differential operator $Q[u]$, compatible with the given
b.c. For
instance, for Neumann b.c.\ and $L<\pi$, the eigenvalues at the saddle $u_0$
are of the form $-1+(\pi k/L)^2$, $k=0,1,2,\dots$, while the eigenvalues at the
local minimum $u_-$ are given by $2+(\pi k/L)^2$, $k=0,1,2,\dots$. Thus
formally, the prefactor in Kramers' law is given by the ratio of infinite
products
\begin{align}
\nonumber
C &= 2\pi\sqrt{\frac{\prod_{k=0}^\infty\abs{-1+(\pi
k/L)^2}}{\prod_{k=0}^\infty\brak{2+(\pi k/L)^2}}} \\
&= 2\pi\sqrt{\frac12\prod_{k=1}^\infty
\frac{1-(L/\pi k)^2}{1+2(L/\pi k)^2}} 
= 2^{3/4}\pi \sqrt{\frac{\sin L}{\sinh(\sqrt{2}L)}}\;.
 \label{SPDE04}
\end{align}
The determination of $C$ for $L>\pi$ requires the computation of ratios of
spectral determinants, which can be done using path-integral techniques
(Gelfand's method, see
also~\cite{Forman1987,McKane_Tarlie_1995,ColinVerdiere1999} for different
approaches to the computation of spectral determinants). The case of periodic
b.c.\ and $L>2\pi$ is even more difficult, because there is a continuous set of
relevant saddles owing to translation invariance, but can be treated as
well~\cite{Stein_JSP_04}. The formal computations of the prefactor have been
extended to the case of bifurcations $L\sim\pi$, respectively $L\sim2\pi$ for
periodic b.c.\ in~\cite{BG09a}. For instance, for Neumann b.c.\ and $L\leqs\pi$,
the expression~\eqref{SPDE04} of the prefactor has to be replaced by 
\begin{equation}
 \label{SPDE05}
C = \frac{2^{3/4}\pi}{\Psi_+(\lambda_1/\sqrt{3\eps/4L})} 
\sqrt{\frac{\lambda_1+\sqrt{3\eps/4L}}{\lambda_1}}
\sqrt{\frac{\sin L}{\sinh(\sqrt{2}L)}}\;,
\end{equation}
where $\lambda_1=-1+(\pi/L)^2$. Unlike~\eqref{SPDE04}, which vanishes in
$L=\pi$, the above expression converges to a finite value of order $\eps^{1/4}$
as $L\to\pi_-$. 

These results have recently been proved rigorously, by considering sequences of
finite-dimensional systems converging to the SPDE as dimension goes to infinity,
and controlling the dimension-dependence of the error terms. The first step in
this direction was made in~\cite{BBM2010} for the chain of interacting particles
introduced in~\cite{BFG06a}, where a Kramers law with uniform error bounds was
obtained for particular initial distributions. Full proofs of the Kramers law
for classes of parabolic SPDEs have then been obtained in~\cite{BG12a}, using
spectral Galerkin approximations for the converging sequence, and
in~\cite{Barret_2012} using finite-difference approximations.


\subsection{The irreversible case}
\label{ssec_cycling}

Does Kramers' law remain valid for general diffusions of the form
\begin{equation}
 \label{irr01}
\6x_t = f(x_t)\6t + \sqrt{2\eps}\6W_t\;, 
\end{equation} 
in which $f$ is not equal to the gradient of a potential $V$? In general, the
answer is negative. As we remarked before, large-deviation results imply that
Arrhenius' law still holds for such systems. The prefactor, however, can
behave very differently than in Kramers' law. It need not even converge to a
limiting value as $\eps\to0$. 

\begin{figure}
\centerline{\includegraphics*[clip=true,height=50mm]{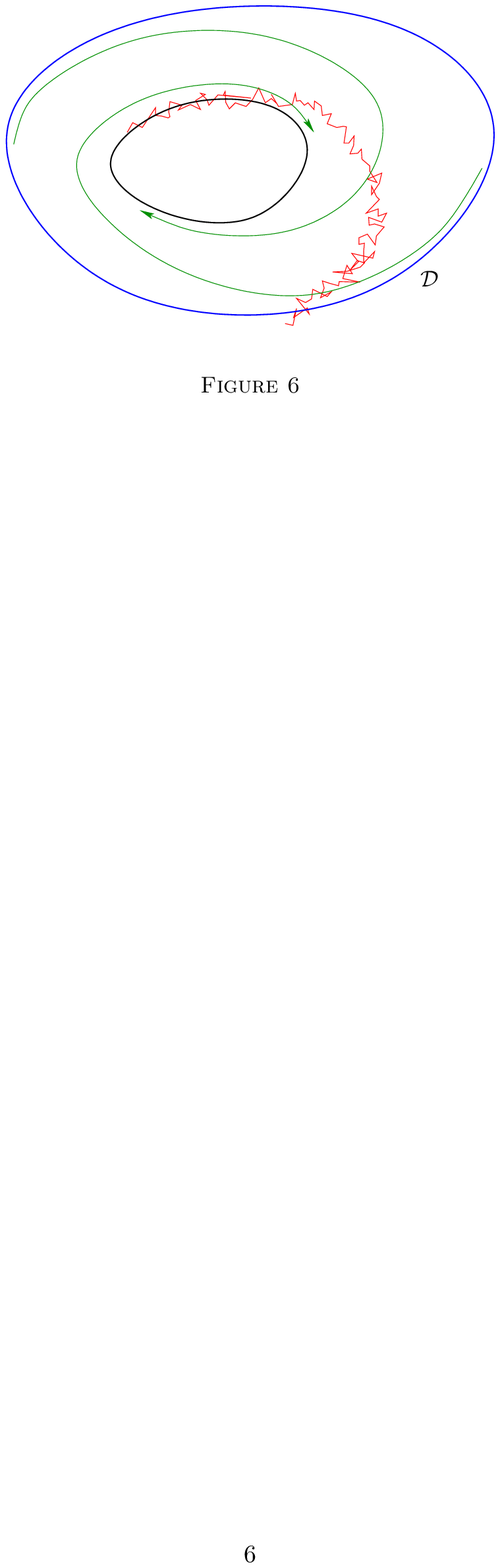}}
 \vspace{2mm}
\caption[]{Two-dimensional vector field with an unstable periodic orbit. The
location of the first exit from the domain $\cD$ delimited by the unstable
orbit displays the phenomenon of cycling. 
}
\label{fig_cycling}
\end{figure}

We discuss here a particular example of such a non-Kramers behaviour, called
\defwd{cycling}. Consider a two-dimensional vector field admitting an unstable
periodic orbit, and let $\cD$ be the interior of the unstable orbit
(\figref{fig_cycling}). Since paths tracking the periodic orbit do not
contribute to the rate function, the quasipotential is constant on
$\partial\cD$, meaning that on the level of large deviations, all points on the
periodic orbit are equally likely to occur as first-exit points. 

Day has discovered the remarkable fact that the distribution of first-exit
locations rotates around $\partial\cD$, by an angle proportional to
$\log\eps$~\cite{Day5,Day6,Day4}. Hence this distribution does not converge to
any limit as $\eps\to0$. 

Maier and Stein provided an intuitive explanation for this phenomenon in terms
of most probable exit paths and WKB-approximations~\cite{MS4}. Even though the
quasipotential is constant on $\partial\cD$, there exists a well-defined path
minimising the rate function (except in case of symmetry-related degeneracies).
This path spirals towards $\partial\cD$, the distance to
the boundary decreasing geometrically at each revolution. One expects that exit
becomes likely as soon as the minimising path reaches a distance of order
$\sqrt{\eps}$ from the boundary, which happens after a number of revolutions of
order $\log\eps$. 

It turns out that the distribution of first-exit locations itself has universal
characteristics:

\begin{theorem}[\cite{BG7,BG12b}]
\label{thm_cycling} 
There exists an explicit parametrisation of $\partial\cD$ by an angle $\theta$
(taking into account the number of revolutions), 
such that the distribution of first-exit locations has a density close to  
\begin{equation}
 \label{cycling01}
p(\theta) = f_{\math{transient}}(\theta)
\frac{\e^{-(\theta-\theta_0)/\lambda\TK}}{\lambda\TK}
P_{\lambda T}(\theta - \log(\eps^{-1}))\;, 
\end{equation} 
where 
\begin{itemiz}
\item	$f_{\math{transient}}(\theta)$ is a transient term, exponentially close
to $1$ as soon as $\theta\gg\abs{\log\eps}$;
\item	$T$ is the period of the unstable orbit, and $\lambda$ is its Lyapunov
exponent;
\item	$\TK = C\eps^{-1/2}\e^{\Vbar/\eps}$ plays the r\^ole of Kramers' time;
\item	the universal periodic function $P_{\lambda T}(\theta)$ is a sum of
shifted Gumbel distributions, given by 
\begin{equation}
 \label{cycling02}
 P_{\lambda T}(\theta) = \sum_{k\in\Z} A(\theta - k\lambda T)\;, 
\qquad
A(x) = \frac12 \e^{-2x-\frac12\e^{-2x}}\;.
\end{equation} 
\end{itemiz}
\end{theorem}

Although this result concerns the first-exit location, the first-exit time is
strongly correlated with the first-exit location, and should thus display a
similar behaviour. 

Another interesting consequence of this result is that it allows to determine
the resi\-dence-time distribution of a particle in the wells of a periodically
perturbed double-well potential, and therefore gives a way to quantify the
phenomenon of stochastic resonance~\cite{BG9}.


\subsection{Some recent developments}
\label{ssec_recent}

Since the first version of this review appeared as a preprint, there have been
quite a number of new results related to the Kramers formula, which shows
that this field of research is still very active. Here is a
non-exhaustive list. 

A number of new approaches analyse the generator of the diffusion using
variational methods from the theory of PDEs. In~\cite{MenzSchlichting12}, Menz
and Schlichting prove Kramers' law for the first nonzero eigenvalue of the
generator using Poincar\'e and logarithmic Sobolev inequalities.
In~\cite{PeletierSavareVeneroni}, Peletier, Savar\'e and Veneroni use
$\Gamma$-convergence to obtain a Kramers law from a more realistic
Kramers--Smoluchowski equation, in which particles are described by their
position and their chemical state. See
also~\cite{ArnrichMielkePeletierSavareVeneroni12} for an approach based on the
Wasserstein distance,
and~\cite{HerrmannNiethammer11,HerrmannNiethammerVelazquez12} for related work. 

A situation where saddles are even more degenerate than in the cases considered
in Section~\ref{ssec_nonquad}, due to the existence of a first
integral, has been considered by Bouchet and Touchette
in~\cite{BouchetTouchette12}. 

In~\cite{CerouGuyaderLelievreMalrieu12}, 
C\'erou, Guyader, Leli\`evre and Malrieu
show that the Gumbel distribution, which we have seen governs the first-exit
distribution through an unstable periodic orbit, also describes the length of
the reactive path, that is, the first successful excursion out of a potential
well. 

Concerning irreversible diffusions, semiclassical analysis has been extended to 
the Kra\-mers--Fokker--Planck equation, which describes the motion of a particle
in a potential when inertia is taken into account, that is, without the
assumption that the motion is overdamped. See for
instance~\cite{HerauHitrikSjostrand08,HerauHitrikSjostrand11}. 

One limitation of the results on SPDEs in Section~\ref{ssec_SPDE} is that the
interval length $L$ is fixed. This implies that the transition states are
stationary solutions that change sign only once or twice per period (depending
on the b.c.). Stationary solutions with more sign changes have a higher energy,
and do not contribute to the transition rate. This is no longer true if
$L=L(\eps)$ grows sufficiently fast as $\eps\to0$. Recent results by Otto, Weber
and Westdickenberg~\cite{OttoWeberWestdickenberg13}, who study the Allen--Cahn
equation in that regime, may help to prove a Kramers formula in that case.
In such systems, metastability may be due to the fact that long time spans are
spent on the stable manifold of transition states, see also the discussion
in~\cite{BeckWayne11}.

\small
\bibliography{../../BFG}
\bibliographystyle{amsalpha}               

\newpage

\tableofcontents

\bigskip\bigskip\noindent
{\small 
Nils Berglund \\ 
Universit\'e d'Orl\'eans, Laboratoire {\sc Mapmo} \\
{\sc CNRS, UMR  7349} \\
F\'ed\'eration Denis Poisson, FR 2964 \\
B\^atiment de Math\'ematiques, B.P. 6759\\
45067~Orl\'eans Cedex 2, France \\
{\it E-mail address: }{\tt nils.berglund@univ-orleans.fr}
}


\end{document}